\documentclass[11pt]{amsart}

\usepackage[utf8]{inputenc}
\usepackage{lmodern}
\usepackage{amsmath,amssymb,amsthm,latexsym,cite,cancel}
\usepackage[small]{caption}
\usepackage{graphicx,wasysym,overpic,tikz,color}
\usepackage{subfigure}
\usepackage{cite}
\usepackage[colorlinks=true,urlcolor=blue,
citecolor=red,linkcolor=blue,linktocpage,pdfpagelabels,
bookmarksnumbered,bookmarksopen]{hyperref}
\usepackage[english, french]{babel}
\usepackage{units}
\usepackage[utf8]{inputenc}
\usepackage{enumitem}
\usepackage[left=2.1cm,right=2.1cm,top=2.71cm,bottom=2.71cm]{geometry}
\usepackage[hyperpageref]{backref}
\usepackage{float}
\usepackage[T1]{fontenc}

\usepackage[colorinlistoftodos]{todonotes}
\newtheorem{theorem}{Theorem}[section]
\newtheorem{definition}[theorem]{Definition}

\newtheorem{lemma}[theorem]{Lemma}
\newtheorem{remark}[theorem]{Remark}

\newcommand{\abs}[1]{\lvert#1\rvert}
\newcommand{\norm}[1]{\lVert#1\rVert}
\newcommand{\red}[1]{\textcolor{red}{#1}}
\newcommand{\blue}[1]{\textcolor{blue}{#1}}
\newcommand{\D}{\mathcal{D}}
\newcommand{\G}{\mathcal{G}}
\newcommand{\K}{\mathcal{K}}
\newcommand{\h}{\mathcal{H}}

\newcommand{\R}{\mathbb{R}}
\newcommand{\C}{\mathbb{C}}
\newcommand{\specialref}[2]{(\hyperref[#2]{#1})$_{c_n}$}
\newcommand{\specialrefc}[2]{(\hyperref[#2]{#1})$_{c}$}

\tikzstyle{nodo}=[circle,draw,fill,inner sep=0pt,minimum size=%
1.5mm]
\tikzstyle{bnodo}=[circle,draw,fill,inner sep=0pt,minimum size=%
2mm]

\numberwithin{equation}{section}

\title[Nonlinear Dirac equations on metric graphs]{Nonrelativistic limit of normalized solutions of nonlinear Dirac equations on noncompact metric graphs with localized nonlinearities}

\author[Z. He]{Zhentao He}

\address[Z. He]{\newline\indent
	School of Mathematics
	\newline\indent
	East China University of Science and Technology
	\newline\indent
	Shanghai 200237, PR China }
\email{\href{mailto:hezhentao2001@outlook.com}{hezhentao2001@outlook.com}}

\author[C. Ji]{Chao Ji}

\address[C. Ji]{\newline\indent
	School of Mathematics
	\newline\indent
	East China University of Science and Technology
	\newline\indent
	Shanghai 200237, PR China }
\email{\href{mailto:jichao@ecust.edu.cn}{jichao@ecust.edu.cn}}

\date{\today}

\begin{document}
\maketitle

\selectlanguage{english}
\begin{center}
    \textbf{ABSTRACT}
\end{center}
\vspace{0.5\baselineskip}

\noindent  In this paper, we study the nonrelativistic limit of normalized solutions for the following nonlinear Dirac equation (NLDE) on noncompact metric graph $\G$ with finitely many edges and a non-empty compact core $\K$
\begin{equation*}
    \D u - \omega u= \chi_\K\abs{u}^{p-2}u,
\end{equation*}
under the constraint  $\int_\G\abs{u}^2\,dx = 1$, where $\D$ is the Dirac operator on $\G$, $u: \G \to \mathbb{C}^2$,  the frequency $\omega \in \mathbb{R}$ is part of the unknowns which arises  as a Lagrange multiplier, $\chi_\K$ is the characteristic function of the compact core $\K$, and $2<p<6$. To the best of our knowledge, this is the first study to  investigate the nonrelativistic limit of normalized solutions to (NLDE) on metric graphs.
\section{Introduction}

Throughout the paper, we consider a connected noncompact metric graph $\mathcal{G} = (\mathrm{V}, \mathrm{E})$, where $\mathrm{E}$ is the set
of edges and $\mathrm{V}$ is the set of vertices. If $\G$ is a metric graph with a finite number of vertices, its compact core $\K$ is defined as the metric subgraph of $\G$ consisting of all the bounded edges. We assume that $\mathcal{G}$ has a finite number of edges and that its compact core $\K$ is non-empty. Each bounded edge $e$ is identified with a closed and bounded interval $I_e = [0, \ell_e]$, $\ell_e > 0$, while each unbounded edge $e$ is identified with a copy of $I_e=\mathbb{R}^+ =  [0, +\infty )$ and is called half-line. A connected metric graph has the natural structure of a locally compact metric space, the metric being given by the shortest distance measured along
the edges of the graph. Moreover, a connected metric graph with a finite number of vertices is compact if and only if it contains no half-lines.

 Quantum graphs (metric graphs equipped with differential operators) arise naturally as simplified models in mathematics, physics, chemistry, and engineering when one considers propagation of waves of various nature through a quasi-one-dimensional (e.g., meso- or nano-scale) system that looks like a thin neighborhood of a graph. For further details on quantum graphs, one may refer to \cite{Be}.
\begin{figure}[H]
	\begin{tikzpicture}[xscale= 0.4,yscale=0.4]
		\draw (-14,3)--(-11,6) ; \draw (0,8)--(-8,6);
		\draw (-11,6)--(-8,6); \draw (-22,7)--(-14,3);
				\draw (0,3)--(-8,3);
		\draw (-11,9)--(-11,6);
		\draw (-14,3)--(-14,6);
		\draw (-11,3)--(-11,6);
		\draw (-8,3)--(-8,6);
		\node at (-14,3) [nodo] {} ; \node at (-14,6) [nodo] {};
		\node at (-11,3) [nodo] {}; \node at (-11,6) [nodo] {} ;
		\node at (-8,3) [nodo] {}; \node at (-8,6) [nodo] {};
		\node at (-11,9) [nodo] {};
  \draw (-22,7.5)  node{$\infty$};
  	\draw (0,8.5)  node{$\infty$};
     	\draw (0,3.5)  node{$\infty$};
   \draw (-11,10.5) circle (1.5);
		\end{tikzpicture}
	\caption{A noncompact graph $\G$ with finitely many edges and a non-empty compact core.}
\end{figure}
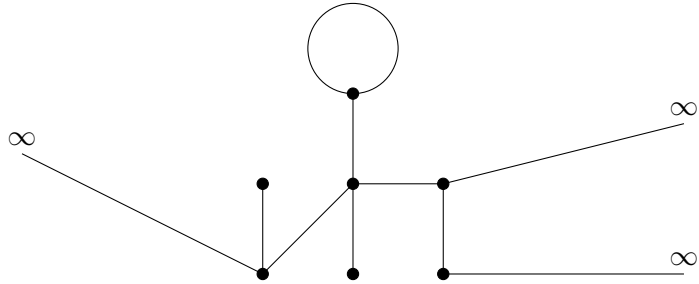

In recent years, a considerable attention has been devoted to the following nonlinear Schr\"odinger equations (NLSE) on metric graphs $\G$
\begin{equation}
\label{eqs}
    -u'' - \lambda u= \abs{u}^{p-2}u.
\end{equation}
In the $H^1$-subcritical ($2 < p < 6$) and $H^1$-critical ($p = 6$) cases, the energy functional is bounded from below and coercive under the mass constraint (requiring the mass to be below a certain threshold when $p=6$), which allows the use of minimization methods to obtain normalized ground states; see \cite{CDS, Do, ABD, Ad2, Ad3, Ad4, Ad5, Do1, NP}. In contrast, for the $H^1$-supercritical case ($p > 6$), the energy functional is no longer bounded from  below on the mass constraint. Moreover, the scaling technique--commonly used in the analysis on $\mathbb{R}^N$---is no longer applicable, and the Pohoz\v{a}ev identity is lacking for metric graphs. To address these challenges, several works have developed new approaches  to obtain normalized solutions; see, for instance, \cite{Bort, Ch, Do0}.

For any noncompact metric graph $\G$ with finitely many edges and a non-empty compact core $\K$, \cite{Gn,No} introduced the following modification of equation \eqref{eqs} which assumes the nonlinearity affects only the non-empty compact core $\K$
\begin{equation}
\label{eqsloc}
    -u'' - \lambda u= \chi_\K\abs{u}^{p-2}u,
\end{equation}
where $\chi_\K$ is the characteristic function of the compact core $\K$. In the $H^1$-subcritical ($2<p<6$) case, the existence, non-existence and multiplicity of normalized solutions to equation \eqref{eqsloc} have been studied in \cite{Se,Se1,Te}.  In the $H^1$-critical ($p=6$) case, the existence or non-existence of normalized ground states to equation \eqref{eqsloc} have been studied in  \cite{Do2, Do3}.  In $H^1$-supercritical ($p>6$) case, the existence of normalized solutions for arbitrary mass was obtained in \cite{Bort} via monotonicity methods, blow-up analysis, and discussions on Morse index type information of Palais-Smale sequences and bound states. Furthermore,  the multiplicity  of normalized solutions to equation \eqref{eqsloc} for $p > 6$ was proved in \cite{Ca} through employing a more general blow-up analysis.

The Dirac operator on metric graphs is denoted by
\begin{equation}
\label{eqdef}
    \D_c := -\imath c\frac{d}{dx} \otimes \sigma_1 +mc^2\otimes\sigma_3
\end{equation}
where $m > 0$ represents the mass of the generic particle of the system and $c > 0$ represents the speed of light, and $\sigma_1, \sigma_3$ are Pauli matrices, i.e.,
\begin{equation}
\label{eqdefsig}
\sigma_1:=\left(\begin{array}{ll}
0 & 1 \\
1 & 0
\end{array}\right) \quad \text { and } \quad \sigma_3:=\left(\begin{array}{cc}
1 & 0 \\
0 & -1
\end{array}\right).
\end{equation}

Consider the following nonlinear Dirac equation (NLDE) on metric graphs $\G$
\begin{equation}\label{eqdinvolvet}
     \imath\partial_t \psi = \D_c\psi + \abs{\psi}^{p-2}\psi,
\end{equation}
where $p > 2$. The physical motivation for such a model primarily comes from solid state physics and nonlinear optics, see \cite{Had,Tr} and references therein.
\cite{Sa} suggests studying the stationary solutions to equation \eqref{eqdinvolvet} on $3$-star graph, that is, solutions in the form $\psi(t, x) = e^{-\imath \omega t}u(x)$,  where $\omega \in \mathbb{R}$ is a parameter. Note that $\psi(t, x) = e^{-\imath \omega t} u(x)$ solves \eqref{eqdinvolvet} if and only if $u$ satisfies the stationary equation
\begin{equation}
\label{eqdexnon}
    \D_c u - \omega u= \abs{u}^{p-2}u.
\end{equation}
In \cite{Bo}, Borrelli, Carlone and Tentarelli considered the case of a localized nonlinearity, that is
\begin{equation}
\label{eqd1}
    \D_c u - \omega u= \chi_\K\abs{u}^{p-2}u.
\end{equation}
More precisely, they proved that for every $\omega \in (-mc^2,mc^2)$, there exists infinitely many (distinct) pairs of bound states (arising as critical points of the corresponding action functional) of frequency $\omega$ of \eqref{eq1}, which converge to the bound states of \eqref{eqs} in the nonrelativistic limit, namely as $c \to +\infty$, for $2<p<6$. For the Cauchy problem associated with equation \eqref{eqdinvolvet} on noncompact metric graphs, and for the existence of solutions to equation \eqref{eqdexnon} on star graphs, see \cite{Bo1}.

So far, there have been very few results concerning normalized solutions of (NLDE) on metric graphs. In \cite{HJ}, He and Ji first studied the existence of normalized solutions to the following (NLDE) on noncompact metric graph $\G$ with finitely many edges and a non-empty compact core $\K$
\begin{equation*}
    \left\{\begin{aligned}
        &\D_c u -\omega  u= a\chi_\K\abs{u}^{p-2}u \quad \text{on every edge } e \in \mathrm{E},\\
        &\int_\G\abs{u}^2\,dx = 1.
    \end{aligned}
    \right.
\end{equation*}
where $a \neq 0$ and $p>2$. In \cite{Ding}, Ding, Yu and Zhao studied the existence of normalized solutions to the following (NLDE) in $\mathbb{R}^3$:
$$
\left\{\begin{aligned}
        &-ic \sum_{k=1}^{3} \alpha_k \partial_k u + mc^2 \beta u -\omega u =P |u|^{p-2} u,\\
        &\int_{R^3}\abs{u}^2\,dx = 1,
    \end{aligned}
    \right.
$$
where $u : \mathbb{R}^3 \to \mathbb{C}^4 $, $\partial_k = \frac{\partial}{\partial x_k}$, $m,c>0$, $P \in (2,\frac{8}{3})$, $P\in L^\infty(\R^3)\backslash\{0\}$ satisfying $P \geq 0$ and some certain assumptions, $\alpha_1, \alpha_2, \alpha_3$, $\beta$ are $4 \times 4$ Pauli-Dirac matrices:
$$
\alpha_k = 
\begin{pmatrix}
0 & \sigma_k \\
\sigma_k & 0
\end{pmatrix},
\quad \beta = 
\begin{pmatrix}
I & 0 \\
0 & -I
\end{pmatrix},
$$
with
$$
\sigma_1 = 
\begin{pmatrix}
0 & 1 \\
1 & 0
\end{pmatrix},
\quad \sigma_2 = 
\begin{pmatrix}
0 & -i \\
i & 0
\end{pmatrix},
\quad \sigma_3 = 
\begin{pmatrix}
1 & 0 \\
0 & -1
\end{pmatrix},
$$
and the frequency $\omega \in \mathbb{R}$ is part of the unknowns which arises  as a Lagrange multiplier.
The authors proved the existence of normalized solutions by combining reduction and perturbation arguments. First, they introduced a perturbation functional $I_{r,\mu}$ associated with the (NLDE) in $\mathbb{R}^3$ and the corresponding reduced functional $J_{r,\mu}$, where $r > 1$ and $\mu>0$ are parameters, and showed that $J_{r,\mu}$ possesses the mountain pass geometry structure. Next, they constructed test-functions from the non-trivial periodic solutions to the following equation
$$
-i\sum_{k=1}^{3} \alpha_k \partial_k u + m \beta u =m u,
$$
whose existence is guaranteed via Fourier transform (see \cite[Lemma 2.2]{Ding}), and used the test-functions to give an upper bound estimate for the minimax level of $J_{r,\mu}$ and the sequence of Lagrange multipliers associated with the Palais-Smale. Third, the authors showed, for $\mu>0$ small enough and  $r>0$ large enough, that there exists a sequence of critical points $\{u_{r,\mu}\}$ of $J_{r,\mu}$ and $\{u_{r,\mu}\}$ possesses a subsequence which converges to a non-trivial solution to (NLDE) in $\mathbb{R}^3$ as $r \to \infty$ and $\mu \to 0^+$. Finally, they established a non-existence result, which enabled them to obtain the existence of normalized solutions.
He and Ji \cite{HJ}, in contrast, overcame the lack of a Fourier transform on general metric graphs—crucial in the construction of test functions in \cite{Ding}—by developing a new approach to build suitable test functions.
This allowed them to estimate the minimax level of the associated functional and to control the sequence of Lagrange multipliers associated with the Cerami sequences.

A natural problem is to study the nonrelativistic limit of normalized solutions for (NLDE) on metric graphs. To the best of our knowledge,  this problem has not been investigated in the existing literature, and the present paper provides the first attempt to address it within the framework of metric graphs. In \cite{Chen}, Chen, Ding, Guo and Wang studied the nonrelativistic limit of normalized solutions for the following (NLDE) in $\R^3$:
$$
\left\{\begin{aligned}
        &-ic \sum_{k=1}^{3} \alpha_k \partial_k u + mc^2 \beta u -\omega u= \Gamma * (K |u|^\kappa) K |u|^{\kappa-2} u + P |u|^{p-2} u,\\
        &\int_{R^3}\abs{u}^2\,dx = 1,
    \end{aligned}
    \right.
$$
where $u : \mathbb{R}^3 \to \mathbb{C}^4 $, $\partial_k = \frac{\partial}{\partial x_k}$, $*$ stands for the convolution, $m,c>0$, $\kappa \in [2,\frac{7}{3})$, $p \in (2,\frac{8}{3}]$, $\Gamma \in C(\R^3\backslash\{0\},(0,+\infty))$ and $K, P \in C^1(\R^3,(0,+\infty))$ satisfying some certain assumptions, $\alpha_1, \alpha_2, \alpha_3$, $\beta$ are $4 \times 4$ Pauli-Dirac matrices, and the frequency $\omega \in \mathbb{R}$ is part of the unknowns which arises  as a Lagrange multiplier. They proved the existence  of normalized solutions $(u_c,\omega_c)$ for (NLDE) in $\R^3$ when the relativistic parameter $c>0$ is sufficiently large and the Lagrange multipliers $(\omega_c)$ satisfying 
\begin{equation}\label{eqlm}
        -\infty<\liminf_{c \to +\infty}(\omega_c -mc^2) \leq \limsup_{c \to +\infty}(\omega_c -mc^2)<0.
\end{equation}
Next, using the Fourier transform, the authors proved that the second components of $(u_c)$ converge to zero, estimated the energy (of the functional associated with the Schr\"odinger equations) of the first components of $(u_c)$, and then, the authors concluded that that the first components of $(u_c)$ converge to $h$ in $H^1(\R^3,\C^2)$ and $(h, \lambda) \in H^1(\R^3,\C^2) \times \R$ is a ground state normalized solution of (NLSE) in $\R^3$. The key step in proving \eqref{eqlm} in \cite{Chen} relies on the following two inequalities:
$$
\min\{mc^2,c\} \norm{u}_{H^\frac{1}{2}(\R^3,\C^4)}^2 \leq \norm{u}_c^2 \leq \max\{mc^2,c\} \norm{u}_{H^\frac{1}{2}(\R^3,\C^4)}^2,
$$
and, for any $q\in (1,\infty)$, there exists $\tau_q>0$ independent of $c$ such that
    $$
    \tau_q\norm{u^\pm}_{L^q}\leq \norm{u}_{L^q}, \quad \forall u \in E_c\cap L^q.
    $$ However, since the proofs of these two inequalities above crucially depend on the Fourier transform, it is difficult to extend them to metric graphs. To overcome these difficulties and establish \eqref{eqlm} on a noncompact metric graph $\G$, we instead employ the inequality 
    \begin{equation}\label{eqsuppu}
        \int_\G \abs{u}^p\,dx \geq \abs{\operatorname{supp}u}^{1-\frac{p}{2}}\left(\int_\G \abs{u}^2\,dx\right)^\frac{p}{2} \quad \text{for all } u \in L^p(\G,\C^2) \text{ with } \operatorname{supp}u \text{ bounded},
    \end{equation}
 where $\abs{\operatorname{supp}u}$ denotes the measure of $\operatorname{supp}u$, together with an inequality derived via spectral theory in Lemma \ref{lemupb}. Nevertheless, due to the reliance on \eqref{eqsuppu}, we can only prove that the first components of $(u_c)$  converge to $g$ in $H^1(\G,\C)$ and $(g,\lambda) \in H^1(\G,\C) \times \R$ is a normalized solution of (NLSE) on $\G$. Whether $(g,\lambda)$ is a ground state normalized solution of (NLSE) on $\G$ remains an open question.

Inspired by \cite{Bo,Chen,HJ,Esteban}, in this paper, we study the nonrelativistic limit of normalized solutions for the following (NLDE) on noncompact metric graph $\G$ with finitely many edges and a non-empty compact core $\K$
\begin{equation}
\label{eq1}
    \D_c u - \omega u= \chi_\K\abs{u}^{p-2}u,
\end{equation}
under the constraint  $\int_\G\abs{u}^2\,dx = 1$, where $2<p<6$, the frequency $\omega \in \mathbb{R}$ is part of the unknowns which arises  as a Lagrange multiplier and $\chi_\K$ is the characteristic function of the compact core $\K$.

Our focus here is on normalized solutions of \eqref{eqd1}, that is, solutions satisfying $\int_\G\abs{u}^2\,dx = \rho$ with $\rho > 0$.  For simplicity, we assume $\rho=1$. With a slight modification, our results can be extended to the other values of $\rho > 0$.

Now, we recall some basic settings of Dirac equations on metric graphs, as introduced in \cite{Bo,Bo1}.
Consistently, a function $u: \G \to \mathbb{C}$ is actually a family of functions $(u_e)$, where $u_e: I_e \to \mathbb{C}$ is the restriction of $u$ to the edge $e$. The usual $L^p$ spaces on the metric graph are defined as follows
as
$$
L^p(\mathcal{G}):=\bigoplus_{e \in \mathrm{E}} L^p(I_e)
$$
with norms
$$
\norm{u}_{L^p(\mathcal{G})}^p := \sum_{e \in \mathrm{E}}\norm{u_e}_{L^p(I_e)}^p, \quad \text { if } p \in[1, \infty), \quad \text { and } \quad \norm{u}_{L^{\infty}(\mathcal{G})}:=\max _{e \in \mathrm{E}}\norm{u_e}_{L^{\infty}(I_e)},
$$
while the Sobolev spaces $H^m(\G)$ are defined as
$$
H^m(\mathcal{G}):=\bigoplus_{e \in \mathrm{E}} H^m(I_e)
$$ with norms

$$
\norm{u}_{H^m(\G)}^2=\sum_{i=0}^m\norm{u^{(i)}}_{L^2(\mathcal{G})}^2.
$$
Accordingly, a spinor $u = (u^1, u^2)^T: \G \to \mathbb{C}^2$ is a family of $2$-spinors
$$
u_e=\binom{u_e^1}{u_e^2}: I_e \longrightarrow \mathbb{C}^2, \quad \forall e \in \mathrm{E},
$$
and thus
$$
L^p(\mathcal{G}, \mathbb{C}^2):=\bigoplus_{e \in \mathrm{E}} L^p(I_e, \mathbb{C}^2),
$$
endowed with the norms
$$
\norm{u}_{L^p(\G, \mathbb{C}^2)}^p := \sum_{e \in \mathrm{E}}\norm{u_e}_{L^p(I_e, \mathbb{C}^2)}^p, \quad \text { if } p \in[1, \infty), \quad \text { and } \quad \norm{u}_{L^{\infty}(\G, \mathbb{C}^2)}:=\max_{e \in \mathrm{E}}\norm{u_e}_{L^{\infty}(I_e, \mathbb{C}^2)},
$$
whereas
$$
H^m(\mathcal{G}, \mathbb{C}^2):=\bigoplus_{e \in \mathrm{E}} H^m(I_e, \mathbb{C}^2)
$$
with the norms
$$
\|u\|_{H^m(\G, \mathbb{C}^2)}^2:=\sum_{e \in \mathrm{E}}\norm{u_e}_{H^m(I_e, \mathbb{C}^2)}^2.
$$
For convenience, we often abbreviate $\norm{u}_{L^p(\G, \mathbb{C}^2)}$ as $\norm{u}_p$ and $\norm{u}_{H^1(\G, \mathbb{C}^2)}$ as $\norm{u}_{H^1}$.

Note that, in the case of (NLSE), the condition that $u$ is continuous on $\G$ is often contained in $H^1(\G)$. However, in the case of (NLDE), we shall keep the conditions of spinors at the vertices separate, as defined in the following.
\begin{definition}[{\cite[Definition 2.3]{Bo}}]
\label{defD}
Let $\mathcal{G}$ be a metric graph and  $m, c>0$. We call the Dirac operator with Kirchhoff-type vertex conditions the operator $\D_c: L^2(\mathcal{G}, \mathbb{C}^2) \rightarrow L^2(\mathcal{G}, \mathbb{C}^2)$ with action
\begin{equation}\label{eqdefde}
    {\D_c}_{\left.\right|_{I_e}} u=\D_{c,e} u_e:=-\imath c \sigma_1 u_e^{\prime}+m c^2 \sigma_3 u_e \quad \forall e \in \mathrm{E},
\end{equation}
$\sigma_1, \sigma_3$ being the matrices defined in \eqref{eqdefsig}, and the domain of $D$ given by
$$
\operatorname{dom}(\D_c):=\{u \in H^1(\mathcal{G}, \mathbb{C}^2): u \text { satisfies \eqref{eqdefdom1} and \eqref{eqdefdom2}}\}
$$
where
\begin{equation}
\label{eqdefdom1}
    u_e^1(\mathrm{v})=u_f^1(\mathrm{v}) \quad \forall e, f \succ \mathrm{v}, \quad \forall \mathrm{v} \in \mathcal{K},
\end{equation}
\begin{equation}
\label{eqdefdom2}
    \sum_{e \succ v} u_e^2(\mathrm{v})_{ \pm}=0 \quad \forall \mathrm{v} \in \mathcal{K},
\end{equation}
"$e \succ v$" meaning that the edge $e$ is incident at the vertex $v$ and $u_e^2(\mathrm{v})_{ \pm}$ stand for $u_e^2(0)$ or $-u_e^2(\ell_e)$ according to whether $x_e$ is equal to $0$ or $\ell_e$ at $v$.
\end{definition}
Note that, it follows from a direct computation that, for all $u \in \operatorname{dom}(\D_c)$,
\begin{equation}\label{eqdcu}
\norm{\D_cu}_2^2=c^2\norm{u'}_2^2+m^2c^4\norm{u}_2^2.
\end{equation}
We now recall some basic properties of the Dirac operator $\D_c$, as discussed in \cite{Bo}. The operator $\D_c$  is self-adjoint on $L^2(\mathcal{G}, \mathbb{C}^2)$, and its spectrum is given by
\begin{equation}
\label{eqsp}
    \sigma{(\D_c)}=( - \infty , - mc^2
] \cup [mc^2, +\infty ).
\end{equation}
Next, we define the associated quadratic form $\mathcal{Q}_{\D_c}$ and its domain $\operatorname{dom}(\mathcal{Q}_{\D_c})$ as follows,
$$
\operatorname{dom}(\mathcal{Q}_{\D_c}):=\left\{u \in L^2(\mathcal{G}, \mathbb{C}^2): \int_{\sigma(\D_c)}|\nu| d \mu_u^{\D_c}(\nu)<\infty\right\}, \quad \mathcal{Q}_{\D_c}(u):=\int_{\sigma(\D_c)} \nu d \mu_u^{\D_c}(\nu),
$$
where $\mu_u^{\D_c}$ denotes the spectral measure associated with $\D_c$ and $u$. Additionally, $\operatorname{dom}(\mathcal{Q}_{\D_c})$ is a closed subspace of
$$
H^{1 / 2}(\mathcal{G}, \mathbb{C}^2):=\bigoplus_{e \in \mathrm{E}} H^{1 / 2}(I_e) \otimes \mathbb{C}^2
$$
with respect to the norm induced by $H^{1 / 2}(\mathcal{G}, \mathbb{C}^2)$. As a consequence of Sobolev embeddings, it follows that
\begin{equation*}
    \operatorname{dom}(\mathcal{Q}_{\D_c}) \hookrightarrow L^p(\G,\mathbb{C}^2) \quad \forall p \in [2, \infty )
\end{equation*}
and that, in addition, the embedding $\operatorname{dom}(\mathcal{Q}_{\D_c}) \hookrightarrow L^p(\K,\mathbb{C}^2)$ is compact, for $p \in [2, \infty )$, due to the compactness
of $\K$. For simplicity, we denote $\operatorname{dom}(\mathcal{Q}_{\D_c})$ by $Y_c$.

It is possible to define the inner product on $Y_c$ by
\begin{equation*}
    (u,v)_c =\Re(\abs{\D_c}^\frac{1}{2}u,\abs{\D_c}^\frac{1}{2}v)_2,
\end{equation*}
where $\Re$ denotes the real part of a complex number, and the induced norm is given by
\begin{equation*}
    \norm{u}_c = (u,u)_c^\frac{1}{2}.
\end{equation*}
From \eqref{eqsp}, for any $u \in Y_c$, we have
$$\norm{u}^2_c \geq mc^2\norm{u}_2^2.$$
According to \eqref{eqsp}, we can define two spectral projectors $P^\pm_c=\frac{1}{2}\left(I-\frac{\D_c}{\abs{\D_c}}\right)$ such that the space $L^2(\G, \mathbb{C}^2)$ possesses the orthogonal decomposition:
\begin{equation}
     L^2(\G, \mathbb{C}^2) = L^-_c \oplus L^+_c,
\end{equation}
where
$$
L^-_c:=\left\{P^-_cu: u \in L^2(\G, \mathbb{C}^2) \right\} \text{ and } L^+_c:=\left\{P^+_cu: u \in L^2(\G, \mathbb{C}^2) \right\},
$$
and $\D_c$ is negative definite on $L^-_c$ and positive definite on $L^+_c$.
Moreover, we can decompose the form domain $Y_c$ as the orthogonal sum of the positive and negative spectral subspaces for the operator $\D_c$,
i.e.,
\begin{equation}
    Y_c = Y^-_c \oplus Y^+_c \text{,\,\, where } Y^\pm_c = L^\pm_c \cap Y_c=P^\pm_c Y_c.
\end{equation}
As a result, every $u \in Y_c$ can be written as
\begin{equation*}
    u = P^+_cu +P^-_cu =: u^+_c + u^-_c.
\end{equation*}

We first prove an existence result when the relativistic parameter $c>0$ is large. 

Denote the longest bounded edge of $\K$ by $e_0$, i.e., $\ell_{e_0} \geq \ell_e$ for all bounded edge $e$. It is well-known that $\sin(\frac{\pi x}{\ell_{e_0}})$ is the eigenfunction corresponding to the first eigenvalue $\lambda_1:=(\frac{\pi}{\ell_{e_0}})^2$ of $-\Delta$ with Dirichlet boundary condition on the bounded interval $[0,\ell_{e_0}]$.
For $2<p<6$, define 
\begin{equation}\label{eqm0}
    m_0(p,\ell_{e_0}):=
\begin{cases}
    0&\text{if }2<p<4,\\
    \frac{p\pi^2}{4}\ell_{e_0}^{\frac{p}{2}-3}\quad &\text{if }p \geq 4.
\end{cases}
\end{equation}
\begin{theorem}\label{th2}
   Let $\G$ be any noncompact metric graph with a non-empty compact core $\K$, $2< p < 6$ and $m >m_0(p,\ell_{e_0})$. Then, there exists $c_0>0$ depending only on $m,p$ and $\G$ such that, for any $c>c_0$, there exists a non-trivial $u_c \in \operatorname{dom}(\D_c)$ (as defined in Definition \ref{defD}) and $\omega_c \in [0, mc^2)$ such that $(u_c,\omega_c)$ is a normalized solution of 
    \begin{equation}\label{eqnlde} 
    \tag*{(NLDE)$_c$}
    \left\{\begin{aligned}
        &\D_c u_c - \omega_c u_c= \chi_\K\abs{u_c}^{p-2}u_c \quad \text{on every edge } e \in \mathrm{E},\\
        &\int_\G\abs{u}^2\,dx = 1.
    \end{aligned}
    \right.
\end{equation}
In addition, $(\omega_c)$ satisfies \eqref{eqlm}.
\end{theorem}
Next, we investigate the nonrelativistic limit of \specialrefc{NLDE}{eqnlde}, namely,  as the relativistic parameter $c \to +\infty$, in which case  the nonlinear Schr\"{o}dinger equation on metric graphs is recovered.

The Sobolev space $\tilde{H}^1(\G,\C)$ is defined as $$
\tilde{H}^1(\G,\C):=\left\{u\in C(\G,\mathbb{C}): u' \in L^2(\G,\C) \text{ and } u\in L^2(\G,\C)\right\}.
$$
\begin{theorem}\label{th3}
    Let assumptions of Theorem \ref{th2} hold, and let $(c_n)$ be a real sequence satisfying $0<c_n \to +\infty$ as $n \to \infty$. Then, up to a subsequence, there exists a sequence $\left((u_{c_n},\omega_{c_n})\right)$ such that  $(u_{c_n},\omega_{c_n})\in\operatorname{dom}(\D_{c_n})\times [0,mc_n^2)$ is a normalized solution of \specialref{NLDE}{eqnlde} and
    $$
    u^1_{c_n} \to g, \quad u^2_{c_n} \to 0\quad \text{in }H^1(\G,\mathbb{C}),
    $$
    as $n\to \infty$, where $(g, \lambda) \in \tilde{H}^1(\G,\mathbb{C}) \times \R$ with $\displaystyle \frac{\lambda}{m}=\lim_{n \to \infty}(\omega_{c_n}-mc_n^2)$ is a normalized solution of the following nonlinear Schr\"odinger equation:
        \begin{equation}\label{eqnlse} \tag{NLSE} 
    \left\{
\begin{array}{ll}
        g'' -\lambda g= 2m\chi_\K\abs{g}^{p-2}g \quad & \text{on every } e \in \mathrm{E},\\
       \displaystyle \sum_{e \succ\mathrm{v}}\frac{dg_e}{dx_e}(\mathrm{v})=0 & \text{for every } \mathrm{v} \in \mathrm{V},\\
        \int_\G\abs{g}^2\,dx = 1,
\end{array}
\right.
\end{equation}
with  $\frac{dg_e}{dx_e}(\mathrm{v})$ standing for $g'_e(0)$ or $-g'_e(\ell_e)$ according to whether $x_e$ is equal to $0$ or $\ell_e$ at $\mathrm{v}$.
\end{theorem}

In \cite{Chen}, for the nonlinear Dirac equations in $\R^3$, since the reduction method in \cite{Chen} used there can only be applied in the $H^{\frac{1}{2}}$-critical/subcritical cases (i.e., $p \in (2,\frac{3}{8}]$ in $\R^3$), the nonrelativistic limit result was established only for these cases. In contrast, in Theorem \ref{th3}, we prove the nonrelativistic limit result on noncompact metric graphs for the $H^{\frac{1}{2}}$-subcritical cases (i.e., $p \in (2,4)$ for metric graphs) for all $m>0$, and for the $H^{\frac{1}{2}}$-critical/supercritical but $H^1$-subcritical cases (i.e., $p \in [4,6)$ for metric graphs) for all $m>m_0(p,\ell_{e_0})$, where $m_0(p,\ell_{e_0})$ is defined in \eqref{eqm0}.

The rest of the paper is organized as follows. In Section \ref{sect2}, we present a Gagliardo-Nirenberg-Sobolev inequality on metric graphs, introduce the perturbation functional and present some useful lemmas. In Section \ref{sect3}, we give the proofs of Theorems \ref{th2} and \ref{th3}.  In Section \ref{sectfr}, we show that the existence result of normalized solution to \eqref{eqnlse} on $\G$ in Theorem \ref{th3} is compatible with the related researches and provide a possibility to prove $(g,\lambda)$, the normalized solution of \eqref{eqnlse} on $\G$ obtained in Theorem \ref{th3}, is a ground state normalized solution of \eqref{eqnlse}.
\section{Preliminaries}\label{sect2}
In this section, we establish a Gagliardo-Nirenberg-Sobolev inequality on metric graphs, introduce the perturbation functional, and provide some fundamental results that will be essential for proving our main theorems.
\subsection{Gagliardo-Nirenberg-Sobolev inequality}
The following lemma is an immediate corollary of $Y_c \subset H^{\frac{1}{2}}(\G,\mathbb{C}^2)$ and the Gagliardo-Nirenberg-Sobolev inequalities on bounded or unbounded intervals, see \cite[Theorem 1.1]{Ha2}. In the following, we abbreviate $\norm{u}_{H^1(\G, \mathbb{C}^2)}$ as $\norm{u}_{H^1}$.
\begin{lemma}\label{lemgnsg}
Let $p \geq 2$, there exist two positive constants $C_{p,\K}$ and $C_{p,\G}$ that depend only on $c$, $m$, $p$ and $\G$ such that
\[
    \int_{\K} \abs{u}^p\,dx \leq C_{p,\K} \norm{u}^{p-2}_c\norm{u}_2^{2}, \quad \forall u \in Y_c,
\]
and
\[
 \int_{\K} \abs{u}^p\,dx \leq \int_{\G} \abs{u}^p\,dx \leq C_{p,\G} \norm{u}^{p-2}_c\norm{u}_2^{2}, \quad \forall u \in Y_c.
\]
\end{lemma}
\begin{lemma}\label{lemgnsgh1}
Let $p \geq 2$. Then, there exist a positive constant $S_{p,\K}$ that depend only on $p$ and $\G$ such that
\[
    \int_{\K} \abs{u}^p\,dx \leq S_{p,\K} \norm{u}_{H^1}^{\frac{p}{2}-1}\norm{u}_2^{\frac{p}{2}+1}, \quad \forall u \in H^1(\G, \mathbb{C}^2),
\]
and there exist a positive constant $S_{\infty,\K}$ that depend only on $\G$ such that
\[
 \norm{u}_{L^\infty(\K, \mathbb{C}^2)} \leq S_{\infty,\K} \norm{u}_{H^1}^\frac{1}{2}\norm{u}_2^\frac{1}{2}, \quad \forall u \in H^1(\G, \mathbb{C}^2).
\]
\end{lemma}
\subsection{The Perturbation Functional}\label{subsecper} 
We define the energy functional $I_{\omega,c}: Y_c \to \mathbb{R}$ associated with equation \eqref{eq1} by
$$
I_{\omega,c}(u) := \frac{1}{2}\norm{u^{+}_c}^{2}_c-\frac{1}{2}\norm{u^{-}_c}^{2}_c-\frac{\omega}{2} \int_{\G}\abs{u}^{2} \,dx-\Psi(u)
$$
for $u=(u^{+}_c+u^{-}_c) \in Y_c$, where
$$
\Psi(u) := \frac{1}{p} \int_{\K}|u|^{p}\,dx.
$$
It follows from standard arguments that $I_{\omega} \in C^{2}(Y_c, \mathbb{R})$. Moreover, by \cite[Proposition 3.1]{Bo}, if $u \in Y_c$ is a critical point of $I_{\omega}$, then $u$ is a solution of problem \eqref{eq1}, and belonging to $\operatorname{dom}(\D_c) \subset H^1(\G, \mathbb{C}^2)$. Thus, $u_e \in C^1(I_e,\mathbb{C}^{2})$ for all $e \in \mathrm{E}$ (see, e.g., \cite[Remark 6 in Chapter 8]{Ha}).

Define $J_c: Y^+_c \rightarrow \mathbb{R}$ by
$$
J_c(v)=I_{0,c}(v+h_c(v))=\max \{I_{0,c}(v+w): w \in Y^-_c\},
$$
where $h_c \in C^1\left(Y^+_c, Y^-_c\right)$ and the critical points of $J_c$ and $I_{0,c}$ are in one-to-one correspondence via the injective map $u \rightarrow u+h_c(u)$ from $Y^+_c$ into $Y_c$, as shown in \cite{Ding,DingX,Ac}.
\section{Proof of Theorem \ref{th2}}\label{sect3}
Define
\begin{equation}\label{eqcssup}
e_c:=\inf_{v \in Y^+_c}\sup _{0 \leq t<|v|_{2}^{-1}} J_c(t v)
\end{equation} 
and
$$S_c(\mu):=\left\{u \in Y_c :\norm{u}^2_2=\mu \right\}.$$
By combining Lemma 2.2 and (2.4) in \cite{HJ}, we know that $e_c >0$. Moreover, by combining the proof of Lemma 2.6 and (2.4) in \cite{HJ}, we obtain the following lemma. 
\begin{lemma}\label{lemma26}
If $e_c\in (0, \frac{mc^2}{2})$, then, there exists $u \in \operatorname{dom}(\D_c)$ such that $I_{0,c}(u)\leq e_c$, and
     \begin{enumerate}[label=\rm(\roman*)]
        \item either $u$ is a critical point of $I_{0,c}$ constrained on $S_c(1)$ with Lagrange multiplier $\omega \in [0,2e_c]$.
\item  or $u$ is a critical point of $I_{0,c}$ constrained on $S_c(\nu)$ for some $0<\nu < 1$ with Lagrange multiplier $\omega =0$.
\end{enumerate}
\end{lemma}
The following lemma provides an upper bound for $\int_{\K}|u^+_c|^{p}\,dx$,  which plays a crucial role in the proof of Theorem \ref{th2}. 
\begin{lemma}\label{lemupb}
    Let $c \geq \frac{1}{m}$. Then, for any $u \in \operatorname{dom}(\operatorname{\D_c})$, we have
    $$
    \int_{\K}|u^\pm_c|^{p}\,dx \leq S_{p,\K} (mc)^{\frac{p}{2}-1}\norm{u}_{H^1}^{\frac{p}{2}-1}\norm{u^\pm_c}_2^{\frac{p}{2}+1}.
    $$
\end{lemma}
\begin{proof}
    Since $\D_c$ is self-adjoint, we know that 
    $$
    \operatorname{dom}(\operatorname{\D_c})=\left\{u \in L^2(\mathcal{G}, \mathbb{C}^2): \int_{\sigma(\D_c)}|\nu|^2 d \mu_u^{\D_c}(\nu)<\infty\right\},
    $$
    hence, $u^\pm_c \in \operatorname{dom}(\operatorname{\D_c})$ and $\norm{\D_cu^+_c}^2_2+\norm{\D_cu^-_c}^2_2=\norm{\D_cu}^2_2$.
    Then, by $c \geq \frac{1}{m}$ and \eqref{eqdcu}, we obtain
    $$
    \norm{u^\pm_c}_{H^1}^2=\norm{(u^\pm_c)'}_2^2+\norm{u^\pm_c}_2^2 \leq \frac{1}{c^2}\norm{D_cu^\pm_c}_2^2\leq \frac{1}{c^2}\norm{D_cu}_2^2 \leq m^2c^2\norm{u}_{H^1}^2.
    $$
    Thus, by Lemma \ref{lemgnsgh1}, we conclude that 
    $$
    \int_{\K}|u^\pm_c|^{p}\,dx \leq S_{p,\K}\norm{u^\pm_c}_{H^1}^{\frac{p}{2}-1}\norm{u^\pm_c}_2^{\frac{p}{2}+1} \leq S_{p,\K} (mc)^{\frac{p}{2}-1}\norm{u}_{H^1}^{\frac{p}{2}-1}\norm{u^\pm_c}_2^{\frac{p}{2}+1}.
    $$
\end{proof}
\begin{remark}
We note that Lemma \ref{lemupb} straightforward in the case of $\mathbb{R}^3$, because for any $q\in (1,\infty)$, there exists $\tau_q>0$ independent of $c$ such that
    $$
    \tau_q\norm{u^\pm}_{L^q}\leq \norm{u}_{L^q}, \quad \forall u \in E_c\cap L^q,
    $$
    see, e.g., \cite[Lemma 2.1]{Chen}. However, since the proof of this result relies on the Fourier transform, it is difficult to extend it to metric graphs.
\end{remark}
We first prove that the condition $e_c < mc^2$ in Lemma \ref{lemupb} holds for sufficiently large $c>0$ and for all $m > m_0(p,\ell_{e_0})$.
\begin{lemma}\label{lemcin}
    $$e_c \leq \frac{mc^2}{2} + \frac{\pi^2}{4m\ell_{e_0}^2}-\frac{1}{p}\ell_{e_0}^{1-\frac{p}{2}} + o(1),\quad \text{as } c \to +\infty.$$
Moreover, if $2<p<4$, then there exists a constant $C_0(m,p,\G)>0$ depending only on $m,p$ and $\G$ such that $e_c\leq mc^2-C_0(m,p,\G)+o(1)$ as $c \to +\infty$.
\end{lemma}
\begin{proof}
Define $\varphi_{e_0}^1:[0,\ell_{e_0}]\to \mathbb{R}$ by $\varphi_{e_0}^1(x)=\sqrt{\frac{2}{\ell_{e_0}}}\sin(\frac{\pi x}{\ell_{e_0}})$.
Define $\varphi^1: \G \to \mathbb{C}$ by
$$\varphi^1(x) = \begin{cases}
    \varphi_{e_0}^1(x)& \text{for }x \in e_0,\\
    0 & \text{else},
\end{cases}$$
and  $\varphi: \G \to \mathbb{C}^2$ by
$$\varphi = \binom{\varphi^1}{0}.$$
Then  $\varphi \in \operatorname{dom}(\D_c)$, $\norm{\varphi}_2=1$ and $b:=\norm{\varphi^\prime}_{2}^2=\norm{\left(\varphi^1\right)^\prime}_{L^2(\G)}^2=(\frac{\pi}{\ell_{e_0}})^2$.

We first estimate  $J_c\left(\frac{\varphi^+_c}{\norm{\varphi^+_c}_2}\right)$.
It is clear that $\varphi \in \operatorname{dom}(\D)$, $(\D- mc^2I)\varphi = \left(0, -\imath c \left(\varphi^1\right)^\prime\right)^T$ and $((\D- mc^2I)\varphi, \varphi)_2 = 0$. Thus, we obtain
\[
c\norm{\left(\varphi^1\right)^\prime}_{L^2(\G)}\norm{\varphi^-_c}_2\geq \abs{((\D- mc^2I)\varphi,\varphi^-_c)_2} = \abs{((\D- mc^2I)\varphi^-_c,\varphi^-_c)_2} = \norm{\varphi^-_c}^2_v + mc^2 \norm{\varphi^-_c}^2_2 \geq 2mc^2\norm{\varphi^-_c}_2^2.
\]
Hence,
\begin{equation}\label{eqphi-}
\norm{\varphi^-_c}^2_c + mc^2 \norm{\varphi^-_c}^2_2 \leq \abs{((\D- mc^2I)\varphi,\varphi^-_c)_2} \leq \frac{1}{2m}\norm{\left(\varphi^1\right)^\prime}_{L^2(\G)}^2=\frac{b}{2m}.
\end{equation}
Since $((\D- mI)\varphi,\varphi^+_c)_2 + ((\D- mI)\varphi,\varphi^-_c)_2 = ((\D- mI)\varphi, \varphi)_2 = 0$, we obtain
\begin{equation}\label{eqphi+}
\norm{\varphi^+_c}^2_c - mc^2 \norm{\varphi^+_c}^2_2 = ((\D- mc^2I)\varphi^+_c,\varphi^+_c)_2 = ((\D- mc^2I)\varphi,\varphi^+_c)_2 \leq \frac{b}{2m}.
\end{equation}
Since $\norm{\varphi}_2^2 =1$ and $\varphi$ is supported on $I_{e_0}$, \eqref{eqphi-} implies
\begin{equation}\label{eqvarphi+c}
\norm{\varphi^+_c}^2_c= \norm{\varphi}^2_c-\norm{\varphi^-_c}^2_c\geq mc^2\norm{\varphi}^2_2-\frac{1}{2m}\norm{\left(\varphi^1\right)^\prime}_{L^2(\G)}^2=mc^2-\frac{b}{2m},
\end{equation}
\begin{equation}\label{eqvarphi+c2}
\left\lvert1-\norm{\varphi^+_c}_2^2\right\rvert=\left\lvert\norm{\varphi}_2^2-\norm{\varphi^+_c}_2^2\right\rvert =  \norm{\varphi^-_c}_2^2\leq \frac{b}{2m^2c^2},
\end{equation}
and
\begin{equation}\label{eqvarphi+ck}
\left\lvert1-\norm{\varphi^+_c}_{L^2(I_{e_0},\mathbb{C}^2)}\right\rvert=\left\lvert\norm{\varphi}_{L^2(I_{e_0},\mathbb{C}^2)}-\norm{\varphi^+_c}_{L^2(I_{e_0},\mathbb{C}^2)}\right\rvert \leq \norm{\varphi^-_c}_{L^2(I_{e_0},\mathbb{C}^2)}\leq \norm{\varphi^-_c}_2\leq \frac{\sqrt{b}}{\sqrt{2}mc}.
\end{equation}
Thus, by Lemma \ref{lemupb} and \eqref{eqphi-}, we have
\begin{equation}\label{eqphi+k}
\begin{aligned}
        \left\lvert{\norm{\varphi}_{L^p(\K,\mathbb{C}^2)}-\norm{\varphi^+_c}_{L^p(\K,\mathbb{C}^2)}}\right\rvert &\leq \norm{\varphi^-_c}_{L^p(\K,\mathbb{C}^2)}\\
    &\leq S_{p,\K}^\frac{1}{p} (mc)^{\frac{1}{2}-\frac{1}{p}}\norm{\varphi}_{H^1}^{\frac{1}{2}-\frac{1}{p}}\norm{\varphi^-_c}_2^{\frac{1}{2}+\frac{1}{p}}\\
    &\leq S_{p,\K}^\frac{1}{p} (mc)^{-\frac{2}{p}}(1+b)^{\frac{1}{4}-\frac{1}{2p}}(\frac{b}{2})^{\frac{1}{4}+\frac{1}{2p}}\\
    &\leq o(1),
\end{aligned}
\end{equation}
as $c \to +\infty$.

 Then, for $t > 0$, from the definition of $h_c$, one has
$$
I_{0,c}(t\varphi^+_c+h_c(t\varphi^+_c)) \geq I_{0,c}(t\varphi^+_c),
$$
which implies that,
\begin{equation}\label{eqhphi}
    \int_{\K}\abs{t\varphi^+_c+h_c(t\varphi^+_c)}^{p}\,dx + \frac{p}{2}\norm{h_c(t\varphi^+_c)}^{2} \leq \int_{\K}\abs{t\varphi^+_c}^{p} \,dx.
\end{equation}
Hence, by \eqref{eqvarphi+c2} and \eqref{eqphi+k}, for $t=\frac{1}{\norm{\varphi^+_c}_2}$, we obtain
\begin{equation}\label{eqhtphi+}
\begin{aligned}
    \left\lVert h\left(\frac{\varphi^+_c}{\norm{\varphi^+_c}_2}\right)\right\rVert_{L^2(I_{e_0},\mathbb{C}^2)}^2&\leq \left\lVert h\left(\frac{\varphi^+_c}{\norm{\varphi^+_c}_2}\right)\right\rVert^{2}_2\\
    &\leq \frac{1}{mc^2}\left\lVert h\left(\frac{\varphi^+_c}{\norm{\varphi^+_c}_2}\right)\right\rVert^{2} \\
    &\leq  \frac{1}{mc^2}{\norm{\varphi^+_c}_2^{-p}}\int_{\K}\abs{\varphi^+_c}^{p} \,dx\\
    &\leq o(1),
\end{aligned}
\end{equation}
as $c \to +\infty$.
From  H\"older's inequality, \eqref{eqvarphi+c2}, \eqref{eqvarphi+ck} and \eqref{eqhtphi+}, it yields that
\begin{equation}\label{eqphi+h}\begin{aligned}
        \int_{\K}\left\lvert\frac{\varphi^+_c}{\norm{\varphi^+_c}_2} + h\left(\frac{\varphi^+_c}{\norm{\varphi^+_c}_2}\right)\right\rvert^{p}\,dx 
        &\geq  \int_{I_{e_0}}\left\lvert\frac{\varphi^+_c}{\norm{\varphi^+_c}_2} + h\left(\frac{\varphi^+_c}{\norm{\varphi^+_c}_2}\right)\right\rvert^{p}\,dx\\
        &\geq \ell_{e_0}^{1-\frac{p}{2}}\left\lVert\frac{\varphi^+_c}{\norm{\varphi^+_c}_2}+h\left(\frac{\varphi^+_c}{\norm{\varphi^+_c}_2}\right)\right\rVert_{L^2(I_{e_0},\mathbb{C}^2)}^p \\
    &\geq \ell_{e_0}^{1-\frac{p}{2}}\left\lvert \frac{\norm{\varphi^+_c}_{L^2(I_{e_0},\mathbb{C}^2)}}{\norm{\varphi^+_c}_2}-\left\lVert h\left(\frac{\varphi^+_c}{\norm{\varphi^+_c}_2}\right)\right\rVert_{L^2(I_{e_0},\mathbb{C}^2)}\right\rvert^p\\
    &\geq \ell_{e_0}^{1-\frac{p}{2}}\left\lvert \frac{\norm{\varphi^+_c}_{L^2(I_{e_0},\mathbb{C}^2)}}{\norm{\varphi^+_c}_2}+o(1)\right\rvert^p\\
    &\geq  \ell_{e_0}^{1-\frac{p}{2}} + o(1),
    \end{aligned}
\end{equation}
as $c \to +\infty$. Thus, using \eqref{eqphi+}, \eqref{eqvarphi+c2} and \eqref{eqphi+h}, we conclude that\red{,}
    \begin{equation}\label{eqjsmallm}
            \begin{aligned}
                J_c\left(\frac{\varphi^+_c}{\norm{\varphi^+_c}_2}\right) =& 
                \frac{\norm{\varphi^+_c}^2_c}{2\norm{\varphi^+_c}^2_2}-\frac{1}{2}\left\lVert h\left(\frac{\varphi^+_c}{\norm{\varphi^+_c}_2}\right) \right\rVert^2_c -  \frac{1}{p}\int_{\K}\left\lvert\frac{\varphi^+_c}{\norm{\varphi^+_c}_2} + h\left(\frac{\varphi^+_c}{\norm{\varphi^+_c}_2}\right)\right\rvert^{p}\,dx\\
                \leq & \frac{mc^2}{2} + \frac{\norm{\varphi^+_c}^2_c - mc^2\norm{\varphi^+_c}^2_2}{2\norm{\varphi^+_c}^2_2}  -\frac{1}{p}\ell_{e_0}^{1-\frac{p}{2}} + o(1)\\
                \leq &\frac{mc^2}{2} + \frac{b}{4m}-\frac{1}{p}\ell_{e_0}^{1-\frac{p}{2}} + o(1)\\
    \end{aligned}
    \end{equation}
as $c \to +\infty$.

Next, define $g(t):=J_c(t\varphi^+_c)$. Then, from \eqref{eqvarphi+c}, \eqref{eqphi+k} and the definition of $J_c$, for $0 \leq t \leq\norm{\varphi^+_c}_{2} ^{-1}$ and sufficiently small $b>0$, we have
$$
\begin{aligned}
g'(t) & =\frac{1}{t}\langle J_c^{\prime}(t\varphi^+_c), t\varphi^+_c\rangle=\frac{1}{t}\langle I_{0,c}'(t\varphi^+_c+h_c(t\varphi^+_c)), t \varphi^+_c+h'(t\varphi^+_c) t\varphi^+_c\rangle \\
& =\frac{1}{t}\langle I_{0,c}^{\prime}(t\varphi^+_c+h_c(t\varphi^+_c)), t\varphi^+_c + h_c(t\varphi^+_c)\rangle \\
& =\frac{1}{t}\left[\norm{t\varphi^+_c}^2_c-\norm{h_c(t\varphi^+_c)}^2_c-\int_{\K} \abs{t\varphi^+_c + h_c(t\varphi^+_c)}^p \,dx\right]\\
& \geq \frac{1}{t}\left[\norm{t\varphi^+_c}^{2}-\int_{\K}\abs{t\varphi^+_c}^{p}\,dx+\frac{p-2}{2}\norm{h_c(t\varphi^+_c)}^{2}\right] \\
& \geq t\norm{\varphi^+_c}^2_c-t^{p-1}\int_{\K} \abs{\varphi^+_c}^{p} \,dx \\
& \geq t\left(\norm{\varphi^+_c}^2_c-\norm{\varphi^+_c}_{2}^{2-p}\int_{\K} \abs{\varphi^+_c}^{p} \,dx\right)\\
& \geq t\left(mc^2-\norm{\varphi}_{L^p(\K,\mathbb{C}^2)}^p+o(1)\right),
\end{aligned}
$$
as $c \to +\infty$, which implies that, $g(t)$ is increasing for $0 \leq t \leq\norm{\varphi^+_c}_{2} ^{-1}$ for all $c>0$ large enough. Then, by \eqref{eqcssup} and \eqref{eqjsmallm}, we conclude that,
\begin{equation}\label{eqec}
    e_c= \sup _{0 \leq t \leq\norm{\varphi^+_c}_{2} ^{-1}} J_c(t\varphi^+_c) \leq J_c(\frac{\varphi^+_c}{\norm{\varphi^+_c}_2}) \leq \frac{mc^2}{2} + \frac{b}{4m}-\frac{1}{p}\ell_{e_0}^{1-\frac{p}{2}} + o(1),\quad \text{as } c \to +\infty.
\end{equation}

Now, let $2<p<4$, the half-lines of the graph $\G$ be denoted by $\h_1, \h_2,...,\h_N$, and the total length of the graph $\K$ be $\abs{\K} = \sum_{e \in \K}\ell_e$. For $a >0$, define $\tilde{\varphi}_{a}^1: \G \to \mathbb{C}$ by
$$
    \tilde{\varphi}_{a}^1(x)= \begin{cases}1 & \text { for } x \in \K\\ \max\{0,1-ax\} & \text { for } x \in \h_i,\quad i = 1,2,...,N, \end{cases}
$$
$\varphi_{a}^1 := \frac{\tilde{\varphi}_{b}^1}{\norm{\tilde{\varphi}_{b}^1}_{L^2(\G,\C)}}$ and define $\varphi_{a}: \G \to \mathbb{C}^2$ by
$$\varphi_{a} = \binom{\varphi_{a}^1}{0}.$$
Then $\norm{\varphi_a}_2=1$, $\norm{\varphi_a}_{L^2(\K,\C)}^2=\frac{\norm{\tilde{\varphi}^1_a}_{L^2(\K)}^2}{\norm{\tilde{\varphi}^1_a}_{L^2(\G)}^2}=\frac{\abs{\K}}{\frac{N}{3a} + \abs{\K}}$, $\norm{\varphi_a}_{L^p(\K,\C)}=\frac{\norm{\tilde{\varphi}^1_a}_{L^p(\K)}}{\norm{\tilde{\varphi}^1_a}_{L^2(\G)}}=\frac{\abs{\K}^\frac{1}{p}}{\left(\frac{N}{3a} + \abs{\K}\right)^\frac{1}{2}}$, and $b_a:=\norm{\varphi_a^\prime}_{2}^2=\norm{\left(\varphi^1_a\right)^\prime}_{L^2(\G)}^2=\frac{\norm{\left(\tilde{\varphi}^1_a\right)^\prime}_{L^2(\G)}^2}{\norm{\tilde{\varphi}^1_a}_{L^2(\G)}^2}=\frac{Na}{\frac{N}{3a} + \abs{\K}}$. By repeating the argument of \eqref{eqphi+h} and replacing  $I_{e_0}$ with $\K$, we deduce that
$$
 \int_{\K}\left\lvert\frac{\varphi^+_{a,c}}{\norm{\varphi^+_{a,c}}_2} + h\left(\frac{\varphi^+_{a,c}}{\norm{\varphi^+_{a,c}}_2}\right)\right\rvert^{p}\,dx \geq \abs{\K}^{1-\frac{p}{2}}\norm{\varphi_a}_{L^2(\K,\C)}^p +o(1)=\frac{\abs{\K}}{\left(\frac{N}{3a} + \abs{\K}\right)^\frac{p}{2}} +o(1),
$$
as $c \to +\infty$, where $\varphi^+_{a,c}=P^+_c\varphi_a$. By repeating the arguments of \eqref{eqjsmallm} and \eqref{eqec}, we conclude
$$
J_c\left(\frac{\varphi^+_c}{\norm{\varphi^+_c}_2}\right) \leq \frac{mc^2}{2} + \frac{b_a}{4m}-\frac{\abs{\K}}{\left(\frac{N}{3a} + \abs{\K}\right)^\frac{p}{2}}+o(1)\blue{,} \quad \text{as }c \to +\infty,
$$
and 
$$
\begin{aligned}
     e_c \leq J_c\left(\frac{\varphi^+_c}{\norm{\varphi^+_c}_2}\right)&\leq \frac{mc^2}{2} + \frac{b_a}{4m}-\frac{\abs{\K}}{\left(\frac{N}{3a} + \abs{\K}\right)^\frac{p}{2}}+o(1)\\
     &= \frac{mc^2}{2} + \frac{Na}{4m\left(\frac{N}{3a} + \abs{\K}\right)}-\frac{\abs{\K}}{\left(\frac{N}{3a} + \abs{\K}\right)^\frac{p}{2}}+o(1)\blue{,}\quad \text{as }c \to +\infty.
\end{aligned}
$$
Since $2<p<4$, we know $\frac{Na}{4m\left(\frac{N}{3a} + \abs{\K}\right)}-\frac{\abs{\K}}{\left(\frac{N}{3a} + \abs{\K}\right)^\frac{p}{2}}<0$ for some $a>0$ small enough. Thus, there exists a constant $C_0(m,p,\G)>0$ depending only on $m,p$ and $\G$ such that $e_c\leq mc^2-C_0(m,p,\G) +o(1)$ as $c \to +\infty$.
\end{proof}
Next, we prove that, if for all sufficiently large $c>0$, there exists a normalized solution $(u_c,\omega_c) \in \operatorname{dom}(\D_c)\times[0,mc^2-\sigma]$ of \specialrefc{NLDE}{eqnlde} with $I_{0,c}(u) < \frac{mc^2}{2}$, then $(u_c)$ are uniformly bounded in $H^1(\G,\C)$ and $mc^2-\omega_c$ is bounded from above as $c \to +\infty$.
\begin{lemma}\label{lembounded}
      Let $2<p<6$. For any $\sigma>0$, if $mc^2>\sigma$ and $(u_c,\omega_c) \in \operatorname{dom}(\D_c)\times[0,mc^2-\sigma]$ is a normalized solution of \specialrefc{NLDE}{eqnlde} such that $I_{0,c}(u) < \frac{mc^2}{2}$, then 
      $$\norm{u_c}_{H^1} < C(\sigma,m,p,\G),$$ 
      where $C(\sigma, m,p,\G)> 0$ is a constant that denpends only on $\sigma,m,p$ and $\G$. Moreover, if for all sufficiently large $c>0$, there exists a normalized solution $(u_c,\omega_c) \in \operatorname{dom}(\D_c)\times[0,mc^2-\sigma]$ of \specialrefc{NLDE}{eqnlde} with $I_{0,c}(u) < \frac{mc^2}{2}$, then
      $$
      \omega_c\geq mc^2 - 2 S_{p,\K}C(\sigma,m,p,\G)^\frac{p-2}{2} + o(1), \quad \text{as }c \to +\infty.
      $$
\end{lemma}
\begin{proof}
It follows from $\langle I_{0,c}^{\prime}(u_{c}), u_{c}\rangle =\omega_c\norm{u_c}_2^2$ and $I_{0,c}(u) < \frac{mc^2}{2}$ that 
\begin{equation}\label{equcp}
    \left(\frac{1}{2}-\frac{1}{p}\right)\int_{\K}\abs{u_{c}}^p\, dx = I_{0,c}(u) -\frac{1}{2}\langle I_{0,c}^{\prime}(u_{c}), u_{c}\rangle <\frac{mc^2}{2}-\frac{\omega_c}{2}\norm{u_c}_2^2.
\end{equation}
By \eqref{eqdcu}, we obtain
\begin{equation}\label{eqc2u'}
c^2\norm{u'_c}_2^2+m^2c^4\norm{u_c}_2^2= \norm{\D_cu_c}_2^2= \omega_c^2\norm{u_c}_2^2+2\omega_c\int_{\K} \abs{u_c}^{p} \,dx+ \int_{\K} \abs{u_c}^{2p-2} \,dx.
\end{equation}
Then, by  $m^2c^4-\omega_c^2 \geq 2mc^2\sigma-\sigma^2>mc^2\sigma>0$, \eqref{equcp} and Lemma \ref{lemgnsgh1}, one has
\begin{equation*}
    \begin{aligned}
    c^2\norm{u'_c}_2^2+mc^2\sigma\norm{u_c}_2^2&<  2\omega_c\int_{\K} \abs{u_c}^{p} \,dx+ \int_{\K} \abs{u_c}^{2p-2} \,dx\\
&\leq  2\omega_c S_{p,\K} \norm{u_c}_{H^1}^\frac{p-2}{2}\norm{u_c}_2^\frac{p+2}{2} + S_{\infty,\K}^{p-2} \norm{u_c}_{H^1}^\frac{p-2}{2}\norm{u_c}_2^\frac{p-2}{2}\int_{\K} \abs{u_c}^{p} \,dx\\
&< 2\omega_c S_{p,\K} \norm{u_c}_{H^1}^\frac{p-2}{2}+ \frac{2pmc^2}{p-2}S_{\infty,\K}^{p-2} \norm{u_c}_{H^1}^\frac{p-2}{2},
\end{aligned}
\end{equation*}
that is,
$$
\min\{1,m\sigma\} \norm{u}^2_{H^1} < \left(\frac{2\omega_c S_{p,\K}}{c^2}+\frac{2pmS_{\infty,\K}^{p-2}}{p-2}\right)\norm{u_c}_{H^1}^\frac{p-2}{2}<\left(2m S_{p,\K}+\frac{2pmS_{\infty,\K}^{p-2}}{p-2}\right)\norm{u_c}_{H^1}^\frac{p-2}{2},
$$
whence, it follows from $2<p<6$ that $\norm{u_c}_{H^1} < C(\sigma,m,p,\G)$, where $$C(\sigma,m,p,\G) := \max\{m^\frac{2}{6-p},\sigma^\frac{2}{p-6}\}\left(2 S_{p,\K}+\frac{2pS_{\infty,\K}^{p-2}}{p-2}\right)^\frac{2}{6-p}.$$

Next, by \eqref{eqc2u'} and Lemma \ref{lemgnsgh1}, we have
$$
\begin{aligned}
    m^2c^4-\omega_c^2=(m^2c^4-\omega_c^2)\norm{u_c}_2^2&\leq c^2\norm{u'_c}_2^2+(m^2c^4-\omega_c^2)\norm{u_c}_2^2\\
    &=  2\omega_c\int_{\K} \abs{u_c}^{p} \,dx+ \int_{\K} \abs{u_c}^{2p-2} \,dx\\
    &\leq 2\omega_c S_{p,\K} \norm{u_c}_{H^1}^\frac{p-2}{2}\norm{u_c}_2^\frac{p+2}{2} + S_{2p-2,\K}\norm{u_c}_{H^1}^{p-2}\norm{u_c}_2^p\\
    &\leq 2\omega_c S_{p,\K}C(\sigma,m,p,\G)^\frac{p-2}{2} + S_{2p-2,\K}C(\sigma,m,p,\G)^{p-2},
\end{aligned}
$$
that is 
$$
\omega_c^2+2\omega_c S_{p,\K}C(\sigma,m,p,\G)^\frac{p-2}{2} + S_{2p-2,\K}C(\sigma,m,p,\G)^{p-2}\geq m^2c^4.
$$
Therefore, $\omega_c \to +\infty$ as $c \to +\infty$, and thus
$$\begin{aligned}
\omega_c &\geq \frac{m^2c^4}{\omega_c} -2 S_{p,\K}C(\sigma,m,p,\G)^\frac{p-2}{2}-\frac{1}{\omega_c}S_{2p-2,\K}C(\sigma,m,p,\G)^{p-2}\\
& \geq mc^2-2 S_{p,\K}C(\sigma,m,p,\G)^\frac{p-2}{2} + o(1),\\
\end{aligned}
$$
as $c \to +\infty$.
\end{proof}
Now, we provide the proof of Theorem \ref{th2}.
\begin{proof}[Proof of Theorem \ref{th2}]
    It follows from Lemma \ref{lemcin} that $e_c\leq mc^2-C_0(m,p,\G) +o(1)$ as $c \to +\infty$ for all $2<p<4$ and $e_c \leq \frac{mc^2}{2} + \frac{b}{4m}-\frac{1}{p}\ell_{e_0}^{1-\frac{p}{2}} + o(1)$ as $c \to +\infty$ for all $4\leq p<6$. Since $m >m_0(p,\ell_{e_0})$, we know that, there exist $c_1>0$ and $\sigma>0$ depending only on $m,p$ and $\G$ such that, for all $c>c_1$, we have $e_c\leq \frac{mc^2-\sigma}{2}$. By Lemma \ref{lemma26}, for any $c>c_1$, there exists $u_c \in \operatorname{dom}(\D_c)$ such that $I_{0,c}(u)\leq e_c<mc^2$, and
     \begin{enumerate}[label=\rm(\roman*)]
        \item either $u_c$ is a critical point of $I_{0,c}$ constrained on $S_c(1)$ with Lagrange multiplier $\omega \in [0,2e_c]$.
\item  or $u_c$ is a critical point of $I_{0,c}$ constrained on $S_c(\nu)$ for some $0<\nu < 1$ with Lagrange multiplier $\omega =0$.
\end{enumerate}
By \cite[Lemma A.2]{HJ}, we know that, there exists $c_0\geq c_1$ depending only on $m,p$ and $\G$ such that, for all $c>c_0$, case (ii) above cannot occur. Thus, for any $c>c_0$, there exists a non-trivial $u_c \in \operatorname{dom}(\D_c)$ and $\omega_c \in [0, mc^2)$ such that $(u_c,\omega_c)$ is a normalized solution of \specialrefc{NLDE}{eqnlde} and $\omega_c \leq mc^2-\sigma$. By Lemma \ref{lembounded}, we complete the proof of Theorem \ref{th2}.
\end{proof}
In a view of \cite[Section 4.2]{Bo}, to prove Theorem \ref{th3}, it suffices to prove that, the solutions $(u_c,\omega_c)$ of \specialrefc{NLDE}{eqnlde}, obtained in the proof of Theorem \ref{th2}, satisfy $(u_c)$ are uniformly bounded in $H^1(\G,\C)$ as $c \to +\infty$, and this is guaranteed by Lemma \ref{lembounded}.
Now, we provide the proof of Theorem \ref{th3}.
\begin{proof}[Proof of Theorem \ref{th3}]
    In the proof of Theorem \ref{th2}, for $m > m_0(p,\ell_{e_0})$ and $c>c_0$, we obtain a sequence of solutions $((u_c,\omega_c))$ of \specialrefc{NLDE}{eqnlde} with $I_{0,c}(u)< \frac{mc^2}{2}$ and 
    $$    -\infty<\liminf_{c \to +\infty}(\omega_c -mc^2) \leq \limsup_{c \to +\infty}(\omega_c -mc^2)<0.
$$
    Then, by Lemma \ref{lembounded}, we know that $(u_c)$ is bounded in $H^1(\G,\C)$. Let $(c_n)$ be a real sequence satisfying $0 < c_n \to +\infty$ as $n \to +\infty$. Then, up to a subsequence, there is $\lambda<0$ such that $\omega_{c_n}-mc_n^2 \to \frac{\lambda}{m}$ as $n \to \infty$. By repeating the argument in \cite[Section 4.2]{Bo}, we conclude that $u_{c_n}^1 \to g$ and $u_{c_n}^2 \to 0$ in $H^1(\G,\C)$ as $n \to \infty$, where $g \in \tilde{H}^1(\G,\C)$ is a solution of the following nonlinear Schr\"odinger equation
    $$
     -g'' -\lambda g= 2m\chi_\K\abs{g}^{p-2}g \quad  \text{on every } e \in \mathrm{E}.
    $$
    Then, $u_{c_n} \to (g,0)^T$ in $H^1(\G,\C^2)$ as $n \to \infty$, hence $\int_{\G}\abs{g}\, dx=1$, which completes the proof of Theorem \ref{th3}.
    
\end{proof}
\section{Further remarks}\label{sectfr}
In this section, we first show that the existence result of normalized solution to \eqref{eqnlse} on $\G$ established in Theorem \ref{th3} is compatible with the related studies, and we then provide a possibility to prove $(g,\lambda)$, the normalized solution of \eqref{eqnlse} on $\G$ obtained in Theorem \ref{th3}, is a ground state normalized solution of \eqref{eqnlse} on $\G$.

Define $$S_\mu:=\left\{u \in\tilde{H}(\G,\C):\norm{u}^2_{L^2(\G,\C)}=\mu\right\},$$
$$E_m(u):=\frac{1}{2}\norm{u'}^2_2 -\frac{2m}{p}\int_\K \abs{u}^p\,dx, \quad \forall u \in \tilde{H}^1(\G,\C),$$
and
$$
\mathcal{E}^\mu_m:=\inf_{u\in S_\mu}E_m(u).
$$
Since
$$
E_m(u)=\frac{(2m)^{\frac{-2}{p-2}}}{2}\left\lVert(2m)^{\frac{1}{p-2}}u'\right\rVert^2_2 -\frac{(2m)^{\frac{-2}{p-2}}}{p}\int_\K \abs{(2m)^{\frac{1}{p-2}}u}^p\,dx=(2m)^{\frac{-2}{p-2}}E_1\left((2m)^{\frac{1}{p-2}}u\right),
$$
we know that the minimization problem $\mathcal{E}_m:=\mathcal{E}_m^1$ has a solution if and only if $\mathcal{E}_1^{\mu_m}$ has a solution, where $\mu_m=(2m)^{\frac{2}{p-2}}$. Then, results in \cite{Se,Se1,Te} can be easily extend to the problem $\mathcal{E}_m$ with little modification.
\\
\\
\\
\noindent \textbf{1.} Let $4\leq p < 6$. For $\varphi^1$ defined as in Lemma \ref{lemcin}, we have
$E_m(\varphi^1) = \frac{1}{2}(\frac{\pi}{\ell_{e_0}})^2-\frac{2m}{p}\int_\K \abs{\varphi^1}^p\,dx \leq \frac{1}{2}(\frac{\pi}{\ell_{e_0}})^2-\frac{2m}{p}\ell_{e_0}^{1-\frac{p}{2}}$. Thus, $\mathcal{E}_m<0$ if $m>m_0(p,\ell_{e_0})$. Then, it follows from \cite[Theorem 3.1]{Te} that $\mathcal{E}_m$ has a solution if $m>m_0(p,\ell_{e_0})$. Therefore, the existence of solutions to \eqref{eqnlse} in Theorem \ref{th3} is compatible with the existence result \cite[Theorem 1.1]{Se1} and the nonexistence result \cite[Theorem 3.2]{Se1}.
\\
\\
\\
\noindent \textbf{2.} If  we can replace \eqref{eqphi+h} by 
$$
\int_{\K}\left\lvert\frac{\varphi^+_c}{\norm{\varphi^+_c}_2} + h_c\left(\frac{\varphi^+_c}{\norm{\varphi^+_c}_2}\right)\right\rvert^{p}\,dx \geq \int_{\K}\left\lvert\frac{\varphi^+_c}{\norm{\varphi^+_c}_2}\right\rvert^{p}\,dx +o(1),
$$
then it is possible to show $(g,\lambda) \in \tilde{H}(\G,\C) \times \R$ obtained in Theorem \ref{th3} is a ground state normalized solution of \eqref{eqnlse} by repeating the argument in \cite[Lemma 4.4]{Chen}, $e_c \leq \frac{mc^2}{2}+\frac{\mathcal{E}_m}{2m}+o(1)$ as $c \to +\infty$, and Theorems \ref{th2} and \ref{th3} hold when $\mathcal{E}_m<0$ ( $\mathcal{E}_m<0$ holds for all $2<p<4$ with $m>0$ or $4 \leq p<6$ with $m>0$ large enough, see \cite{Te}). 

To this end, we know that $\frac{\varphi^+_c}{\norm{\varphi^+_c}_2} + h_c\left(\frac{\varphi^+_c}{\norm{\varphi^+_c}_2}\right) \in \operatorname{dom}(\D_c)$ by the self-jointness of $\D_c$, but it seems difficult to give an upper bound to $H^1$-norms of $\frac{\varphi^+_c}{\norm{\varphi^+_c}_2} + h_c\left(\frac{\varphi^+_c}{\norm{\varphi^+_c}_2}\right)$. An alternative method is to prove 
$$
\min\{mc^2,c\} \norm{u}_{H^\frac{1}{2}(\G,\C^2)}^2 \leq \norm{u}_c^2,
$$
however, in $\R^3$, this inequality is proved by Fourier transform, which is absent on metric graphs, so it remains unclear whether the inequality holds on metric graphs, where only real interpolation theory or spectrum theory in \cite{Bo} are available.

\subsection*{Conflict of interest}

The authors declare no conflict of interest.

\subsection*{Ethics approval}
 Not applicable.

\subsection*{Data Availability Statements}
Data sharing not applicable to this article as no datasets were generated or analysed during the current study.

\subsection*{Acknowledgements}
 C. Ji was partially supported by National Natural Science Foundation of China (No. 12571117).

  \end{document}